\documentclass[12pt, a4paper]{article}
\usepackage[latin1]{inputenc}
\usepackage[english]{babel}
\usepackage{amsmath,verbatim,amsthm,mathrsfs,stmaryrd}
\usepackage[english]{babel}
\usepackage{indentfirst}
\usepackage{amstext}
\usepackage{enumerate}
\usepackage{amsfonts}
\usepackage{textcomp}
\usepackage{amssymb}
\usepackage[dvips]{graphicx}
\usepackage{setspace}
\usepackage{color}
\usepackage{graphicx}
\graphicspath{ {./images/} }
\usepackage[all]{xy}
\newcommand{\field}[1]{\mathbb{#1}}

\newcommand {\R}{\mathbb{R}}
\newcommand {\N}{\mathbb{N}}

\newcommand {\Z}{\mathbb{Z}}

\newcommand{\ZZ}{\field{Z}}
\newcommand{\NN}{\field{N}}

\newcommand{\bgln}{\begin{eqnarray}} 
\newcommand{\egln}{\end{eqnarray}}
\newcommand{\bgl}{\begin{equation}} 
\newcommand{\egl}{\end{equation}}

\newcommand{\al} {\alpha}       
        
\newcommand{\ga} {\gamma}

      \newcommand{\La}{\Lambda}

\newcommand{\ro}{\rho}

\newcommand{\si} {\sigma}

\newtheorem{teorema}{Theorem}[section]
\newtheorem{lemma}[teorema]{Lemma}
\newtheorem{corollary}[teorema]{Corollary}
\newtheorem{definition}[teorema]{Definition}
\newtheorem{proposition}[teorema]{Proposition}
\newtheorem{example}[teorema]{Example}
\newtheorem{remark}[teorema]{Remark}
\newtheorem{theorem}[teorema]{Theorem}

\begin{document}

\title{Entropy for partial actions of $\Z$}


\author{Alexandre Baraviera, Ruy Exel, Daniel Gon\c{c}alves\footnote{Partially supported by Conselho Nacional de Desenvolvimento Cient\'ifico e Tecnol\'ogico (CNPq) grant numbers 304487/2017-1 and 406122/2018-0  and Capes-PrInt grant number 88881.310538/2018-01 - Brazil.}, \\ Fagner B. Rodrigues and Danilo Royer}


\date{}
\maketitle


\begin{abstract} In this paper we introduce the definition of topological entropy for a partial $\mathbb{Z}$-action
on metric spaces.
We show that the definition of partial entropy is an extension of the definition
of topological entropy for a $\mathbb Z$-action. We also prove that the partial
topological entropy is concentrated on the non-wandering set.
\end{abstract}

\section{Introduction}
The purpose of this note is to explore the concept of entropy in the
partial action context. Topological entropy is an invariant 
of a dynamical system (action of the additive group $\ZZ$)
that is constant for topologically equivalent actions, and
a considerable effort has been made during the recent past
in order to successfully extend it for more general group actions,
see for example
\cite{Bis,BiU,L.Bowen,L.Bowen2,L.Bowen3,Bufetov,Bufetov2,CRV,GLW,Ghys,Katok,VaFa}.

A discrete time dynamical system consists of a nonempty set $X$
and a map
$f:X\to X$.
 For a given dynamical system, the main goal is to understand
the discrete time evolution of  points $x\in X$, say,
to study some properties of the set called the orbit, defined
as $\mathcal O(x)=\{f^n(x):n\in \NN \}$ (or,  in the case where
the map $f$ is invertible, $\mathcal O(x)=\{f^n(x):n\in \ZZ \}$).
If the map $f$ is a homeomorphism we can see that the dynamical system
corresponds to a continuous
action of the additive group $\ZZ$ on the set $X$, both having the
same orbits.

 In an intuitive view topological entropy counts, in some sense,
the number of distinguishable orbits over time,
thereby providing an idea of how complex the orbit structure of a system is. For example, entropy distinguishes a dynamical system where points that are close together remain close
from a dynamical system in which groups of points move farther away from each other.


If a homeomorphism $f : A \to f(A)$ is such that $f(A) \not\subset A$, even though the
orbit of a point may not be well-defined, we can consider pieces of orbits,
as long as iterating by $f$ makes sense.
In this setting,
partial group actions arise naturally; however, an extension of the concept
of topological entropy for this setting is still missing. 
In this paper, we consider partial actions of the additive group
$\ZZ$ and study properties it shares with the usual topological entropy associated with dynamical systems.




This work is organized as follows: after reviewing some basic concepts about partial actions
(section 2), we define a partial topological entropy for $\ZZ$ partial actions acting on clopen sets of  metric spaces (section 3), establish some of
its properties (section 4),
 and
then we present  a result analogous to Bowen's for actions
of $\mathbb Z$, showing that the partial entropy
is concentrated on the (partial) non-wandering set (section 5).


\section{Partial action of the additive group $\ZZ$}
\begin{definition}\label{def:partial-action}
Let $X$ be a topological space.
A partial action of $\ZZ$ on $X$ is a pair
$$
\al=(\{X_n\}_{n\in\ZZ},\{\al_n\}_{n\in\ZZ}),
$$
consisting of a collection $\{X_n\}_{n\in\ZZ}$ of open subsets of $X$
and a collection of homeomorphisms $\{\al_n\}_{n\in\ZZ}$

 $$
      \al_n \colon X_{-n} \to X_n  \;,\; n \in \ZZ
 $$
such that \\
$(i)$ $X_0=X$ and the map $\al_0$ is the identity map on $X$, \\
$(ii)$ $\al_n(X_{-n}\cap X_{m}) = X_n \cap X_{n+m}$ and \\
$(iii)$ $\al_{n}(\al_m(x))=\al_{n+m}(x)$ for all $x\in X_{-m}\cap X_{-m-n}$.
\end{definition}

Partial actions arise in multiple settings, see \cite{Misha} for an overview. One possible way to produce examples of partial actions is by  restricting an action to a subset:
\begin{example}
\emph{(The restriction of a global $\ZZ$-action)}
Let $\beta:\mathbb Z\times X\to X$ be a global action and  be $Y$ be an open subset
of $X$. Consider $\alpha$ as the restriction of $\beta$ to $Y$, that is: $Y_t = Y\cap\beta_t(Y)$, and
$\alpha_t : Y_{-t} \to Y_ {t}$ such that $\al_t(x) = \beta_t(x)$, $\forall t \in\ZZ$, $x\in X_{-t}$.
 It is not difficult  to verify that $\alpha$ is a partial action on $Y$.
 \end{example}

 A concrete example of the idea above is a global action
 on the sphere $S^2$ that can be restricted to an open subset $Y$
 (homeomorphic to an open square) in what is well known as the horseshoe,
 an example introduced by Smale in the 60's and that is a
 paradigm of the class of hyperbolic dynamical systems,
 see for example \cite{Brim}.
\begin{example}
\emph{(The horseshoe)} Here we just present the dynamics
restricted to the open unitary square of the plane, say
$Q = (0, 1) \times (0, 1)$. Then
$$
   f \colon (0, 1) \times (1/6, 2/6) \to (4/6, 5/6) \times (0, 1) \quad  (x, y) \mapsto ((x+4)/6, 6y-1)
$$
and
$$
   f \colon (0, 1) \times (4/6, 5/6) \to (1/6, 2/6) \times (0, 1) \quad (x, y) \mapsto ((x+1)/6, 6y-4)
$$
There exists a compact set $\La = \bigcap_{n \in \ZZ} f^n(Q)$ that is
homeomorphic to the symbolic space $\{0, 1 \}^{\ZZ}$ and where $f$ is
conjugate to the dynamical system $\si$ known as the shift, defined as
$(\si(x))_k = x_{k+1}$, where $x \in  \{0, 1 \}^{\ZZ}$ is $x= (\ldots, x_{-1}, x_0, x_1, x_2, \ldots)$.
This map has many interesting dynamical features: for example, it  has periodic points with any given period, has
dense orbits, and is expansive; for a more complete description, see \cite{Brim}.
\end{example}

 In \cite{MR1957674} Abadie  shows that the previous procedure (i.e., to obtain
 a partial action as a restriction of an action) is indeed very general
 since any partial action $(\alpha, G, X)$ can be, in some sense, made global in a suitable
 space. Moreover, he proves that the globalization is unique up to an equivalence
 map, which is the  identity map when restricted to $X$. By equivalence what is meant is exactly
 the following:

\begin{definition}\label{pqfotocoordenadores}
Let $G$ be a group, and suppose that, for each $i=1,2$ we are given a
partial action $(\{X^i_g\}_{g\in G},\{\theta_g^i\}_{g\in G})$ of $G$
on a topological space $X^i$. A continuous map
$$
\phi:X^1\to X^2
$$
will be said $G$-equivariant when, for all $g\in G$, one has that
\\
$(i)$ $\phi(D^1_g)\subset D_g^2$, and \\
$(ii)$ $\phi(\theta_g^1(x))=\theta_g^2(\phi(x))$ for all $x\in D^1_{g^{-1}}$.\\
If moreover $\phi$ is bijective and $\phi^{-1}$
is also $G$-equivariant, we will say that $\phi$
is an equivalence of partial actions. If such an equivalence exists, we will say
that $\theta^1$ is equivalent to $\theta^2$.
\end{definition}

 For the convenience of the reader, we recall Abadie's construction of the globalization in the
 particular case of a $\ZZ$-partial action.

   First, consider the (global) action
$$
   \ga \colon \ZZ \times \ZZ \times X  \to \ZZ \times X
$$
given by $\ga_t(s, x) = (t+s, x)$ for any $t, s \in \ZZ$ and $x \in X$. Now
consider the equivalence relation defined by the following identification:
$$
  (r, x) \sim (s, y) \Leftrightarrow x \in X_{s-r} \; \text{and} \; \al_{r-s}(x) = y
$$

Then take $X^e= (\ZZ \times X) / \sim $ with the quotient topology; the action of $\ga$ preserves
the equivalence classes, and so defines a continuous
global action of $\ZZ$ on
$X^e$ that we denote by $\al^e$. The set $X^e $ is the space of orbits
of $X$ under $\al^e$.

\begin{remark}
By the $\ZZ$-equivariance of the equivalence map we have that if a partial action has
two different globalizations then they
correspond to two topologically conjugate dynamical systems. In particular, this implies that
they have the same topological entropy.
More details on partial actions and their enveloping actions can be found, for example, in \cite{Ruy,EGG}.
\end{remark}

\section{The entropy}

From now on we assume that $\al=(\{X_n\}_{n\in \Z},\{\al_n\}_{n\in \Z})$ is a partial action on a metric space $(X,d)$. We denote the set of natural numbers $\{1,2, \ldots\}$ by $\N$ and the set of non-zero integers by $\Z^*$. 
 
  For $n\in\NN$ and $x\in X$ we
define the following sets
\begin{align}\label{def:I_n}
 I_n(x)=\{i\in\{0,1,\dots, n-1\}:x\in X_{-i}\} \text{ and }\\
 I_{-n}(x)=\{i\in\{0,-1,\ldots, -n+1\}:x\in X_{-i}\}.
 \end{align}
Since $X_0=X$, $I_n(x)\not=\emptyset$ for any $n\in\mathbb{Z}^*$ and $x\in X$.
Then we define a map $d_n:X\times X\to[0, +\infty)$, for $n\in\mathbb{Z}^*$, by
 $$
 d_n(x,y)=\max_{i\in I_n(x)\cap I_n(y)} d(\alpha_i(x),\alpha_i(y)).
 $$
We notice that if $I_n(x)\cap I_n(y)=\{0\}$  then $d_n(x,y)=d(x,y)$.
Another important remark is the fact that $d_n$ may not be a metric, because it is not possible
to guarantee the triangle inequality, as the next example shows. 

\begin{example}\label{contraexemplo}
Consider the partial action of the integers on the interval $[0,1]$ given by $X_{-n}=[0,\frac{1}{2^n})$, $X_{n}=[0,1)$, and $\alpha_n (x) = 2^n x$, for all interger $n\geq 1$. Notice that $d_3(0,\frac{6}{25})= \frac{24}{25}$. On the other hand $d_3(0,\frac{1}{4})+d_3(\frac{1}{4},\frac{6}{25}) = \frac{1}{2}+\frac{1}{50}= \frac{13}{25}$. Hence the triangle inequality is not satisfied for $d_3$. 
\end{example}

Even if we require the domains of the partial action to be clopen, $d_n$ may still fail the triangle inequality, as we show in the example below.

\begin{example}\label{siri} Let $X=\prod \{0,1\}$, with $d((x_n),(y_n)) = \frac{1}{2^i}$, if $x_n=y_n$ for $1\leq n \leq i$ and $x_{i+1}\neq y_{i+1}$ and $d((x_n),(y_n)) = 0$ if $x_n=y_n$ for all $n$. We will define a partial action on $X$. For this, let $[0^n]=\{(x_j)\in X: x_j = 0, \ 1\leq j \leq n\}$ and define $X_{-n}=[0^{3n}]$ and $X_n=X$, for every $n\in \N$. Of course $X_0=X$. The homeomorphism $\alpha_n:X_{-n} \rightarrow X_n$, for $n\in \N$, is defined by $\alpha_n((x_i)_{i\in \N})= (x_{i+3n})_{i\in \N} $. Now, let $x=000111\ldots$, $y=0000111\ldots$ and $z=00111\ldots$. Then $d_2(x,y)=1$ but $d_2(x,z)= \frac{1}{4}$ and $d_2(z,y)= \frac{1}{4}$.

\end{example}

\begin{remark} Notice that, although the triangle inequality does not hold in general, it holds for those $x,y$ and $z$ $\in $ $X$ such that $I_n(x)=I_n(y)=I_n(z)$.
 \end{remark}

\begin{definition} Let $(X,d)$ be a metric space, $\al=(\{X_n\}_{n\in \Z},\{\al_g\}_{n\in \Z})$ be a partial action, and let $K$ be a compact subset of $X$. Given $\varepsilon>0$ and $n\in\Z^*$ we say that:
\begin{itemize}
\item a subset $A\subset K$ is $(n,\varepsilon)$-separated
if for every $x\neq y \in A$, either $I_n(x)\neq I_n(y)$ or $d_n(x,y)\geq\varepsilon$ in case $I_n(x)=I_n(y)$.
\item a subset $B\subset K$ is $(n,\varepsilon)$-spanning (for $K$)
if for every $x\in K$ 
there exists $y\in B$ such that $I_n(x)=I_n(y)$ and $d_n(x,y)<\varepsilon$.
\item an open cover $\mathcal U$ of $K$ is an $(n,\varepsilon)$-cover if
for any $U\in\mathcal U$ and any $x,y\in U$, we have $I_n(x)=I_n(y)$ and $d_n(x,y)<\varepsilon$.
\end{itemize}
\end{definition}

\noindent\textbf{\emph{Standing Hypothesis:}} From now on, all partial actions in this paper are assumed to act on clopen sets.

\vspace{0.5pc}

With the clopen hypothesis above we can prove the following key lemmas.

\begin{lemma}
Let $\al=(\{X_n\}_{n\in \Z},\{\al_n\}_{n\in \Z})$ be a partial action on a metric space $X$. Then, for a fixed $x\in X$ and $n\in \Z$, the set 

$$\{y\in X: I_n(y)=I_n(x)\}$$ is clopen.
\end{lemma}

\begin{proof} Fix $n\in \Z$. Let us suppose $n\geq 1$. Define a map $I:X\rightarrow \{0,1\}^{n-1}$ by $I(y)=(z_0,...,z_{n-1})$, where $z_m=1$ if $m\in I_n(y)$, and $z_m=0$ otherwise. We show that $I$ is continuous (where we consider $\{0,1\}^{n-1}$ with the product topology): let $(y_k)_{k\in \N}$ be a sequence converging to $y\in X$. Let $0\leq i\leq n-1$ be such that $y\in X_{-i}$. Since $X_{-i}$ is open, there exists $k_{i_1}\in \N$ such that $y_k\in X_{-i}$ for $k\geq k_{i_1}$. For $0\leq i\leq n-1$, with $y\notin X_{-i}$, there also exists $k_{i_2}$ such that $y_k\notin X_{-i}$ for each $k\geq k_{i_2}$, since $X\setminus X_i$ is open. Hence there exists $k_0\in \N$ such that for $k>k_0$, and for each $0\leq i\leq n-1$, $y_k\in X_{-i}$ if and only if $y\in X_{-i}$. This shows that $I(y_k)\rightarrow I(y)$.

Now, fix $x\in X$, and notice that $$\{y\in Y:I_n(y)=I_n(x)\}=\{y\in Y:I(y)=I(x)\}=I^{-1}(I(x)).$$
Since $I(x)$ is clopen and $I$ is continuous then $\{y\in Y:I_n(y)=I_n(x)\}$ is clopen.

The case $n<0$ follows analogously.

\end{proof}

\begin{lemma}\label{AvaiVenceuUma!}
Let $\al=(\{X_n\}_{n\in \Z},\{\al_n\}_{n\in \Z})$ be a partial action on a metric space $X$, and let $K$ be a compact subset of $X$.
\begin{enumerate}
\item For each $n\in \ZZ^*$ and $\varepsilon>0$ there exists a finite $(n,\varepsilon)$-cover of $K$.

\item For any $\varepsilon>0$ and $n\in \mathbb \ZZ^*$  there exist a finite $(n,\varepsilon)$-spanning set of $K$.
\end{enumerate}
\end{lemma}

\begin{proof} 
To prove the first item let $n\in\ZZ^*$ and $\varepsilon>0$. Since $K$ is compact, to show the existence of a finite $(n,\varepsilon)$-cover of $K$ it is enough to show that there exists some $(n,\varepsilon)$-cover of $K$. Define, for each
$x\in K$, the sets
$$  U(x,n,\varepsilon):=\{y\in X: I_n(y)=I_n(x) \mbox{ and }d_n(x,y)<\varepsilon\}.$$
Notice that $K\subseteq\bigcup\limits_{x\in K}U(x,n,\varepsilon)$ and $U(x,n,\varepsilon)$ is open
for each $x\in K$, since \[U(x,n,\varepsilon)= (\bigcap_{i\in I_n(x)} \alpha_{i^{-1}}\left( \mathcal{B}(\alpha_i(x),\varepsilon) \cap X_i \right))\cap \{y\in X:I_n(y)=I_n(x)\},\]
where $\mathcal{B}(\alpha_i(x),\varepsilon)$ denotes the $d-$ball of radius $\varepsilon$ centered at $\alpha_i(x)$.
Since $K$ is compact, there exists a finite subcover
$(U(x_i,n,\varepsilon))_{i=1}^\kappa$ of the cover $(U(x,n,\varepsilon))_{x\in K}$.

To prove the second item, notice that the set $\{x_1,...,x_\kappa\}$ (with $x_i$ as before) is a finite $(n,\varepsilon)$-spanning set.
\end{proof}

Let $x\in X$ and $n\in\ZZ$.
Based on the proof above we define the \emph{partial dynamical ball} of radius $\varepsilon$ as the set
\begin{align}\label{eq:spanning-cover}
U(x,n,\varepsilon):=(\bigcap_{i\in I_n(x)} \alpha_{i^{-1}}\left( \mathcal{B}(\alpha_i(x),\varepsilon) \cap X_i \right))\bigcap \{y\in X:I_n(y)=I_n(x)\}.
\end{align}
Notice that given $\varepsilon>0$ and $n\in\mathbb N$, the family $(U(x,n,\varepsilon))_{x\in X}$ defines
an open cover of $X$.

\begin{definition} Let $\al=(\{X_n\}_{n\in \Z},\{\al_n\}_{n\in \Z})$ be a partial action on a metric space $X$ and let $K$ be a compact subset of $X$. We define:
\begin{itemize}

\item $\mbox{sep}(n, \varepsilon, \alpha,K )$ as being the maximum cardinality of an $(n,\varepsilon)$-separated set in $K$,
\item $\mbox{span}(n, \varepsilon, \alpha,K )$ as being the minimum cardinality
of an $(n,\varepsilon)$-spanning set in $K$,
\item $\mbox{cov}(n,\varepsilon,\alpha,K)$ as being the minimum
cardinality of an $(n,\varepsilon)$-cover of $K$. 
\end{itemize}
\end{definition}

\begin{remark} It follows from the second item of Lemma~\ref{AvaiVenceuUma!} that $\mbox{cov}(n,\varepsilon,\alpha,K)$ is finite, for each $n\neq 0$ and $\varepsilon>0$.

\end{remark}

Next, we describe a relation between the numbers defined above, which will be used in the characterization of the partial entropy.

\begin{lemma}\label{lemma:key-lemma} Let $\al=(\{X_n\}_{n\in \Z},\{\al_n\}_{n\in \Z})$ be a partial action on a metric space $X$, and let $K$ be a compact subset of $X$. Then, for any $\varepsilon>0$ and any $n\in\Z^*$,
$$\mbox{cov}(n, 2\varepsilon, \alpha,K )\leq \mbox{span}(n, \varepsilon, \alpha,K )\leq \mbox{sep}(n, \varepsilon, \alpha,K )\leq\mbox{cov}(n, \varepsilon, \alpha,K ).$$
\end{lemma}

\begin{proof}
Let $A\subseteq K$ be a $(n,\varepsilon)$-spanning set. Consider the family
$\mathcal U=\{U(x,n,\varepsilon)\}_{x\in A}$, where $U(x,n,\varepsilon)$
is given by \eqref{eq:spanning-cover}.
 We claim that
$\mathcal{U}$ is an $(n,2\varepsilon)$-cover of $K$.  In fact, since $A$ is an
$(n,\varepsilon)$-spanning set, $\mathcal{U}$ covers $K$. If $z_1,z_2\in U\in\mathcal{U}$,
then there exists $x\in A$ so that $U=U(x,n,\varepsilon)$ and hence
$I_n(z_1)=I_n(z_2)=I_n(x)$. This implies that we can compute
$\alpha_j(z_i)$ for each $j\in I_n(x)$, $i=1,2$.
As a consequence, we get that
$$
d_n(z_1,z_2)\leq d_n(z_1,x)+d_n(x,z_2)<2\varepsilon.
$$
This proves the claim and implies that $\mbox{cov}(n, 2\varepsilon, \alpha,K )\leq \mbox{span}(n, \varepsilon, \alpha,K )$.

For the third inequality, let  $B\subseteq K$ be an $(n, \varepsilon)$-separated set, and $\mathcal{U}$ be a finite $(n,\varepsilon)$-cover. 
If the cardinality of $B$ is greater than 
the cardinality of $\mathcal{U}$
then there exist $U\in\mathcal{U}$
which contains, at least, two distinct elements of $B$. As the $d_n$-diameter of $U$ is less than $\varepsilon$, this contradicts the fact that $B$ is $(n,\varepsilon)$-separated. Hence the cardinality of $B$ is less or equal to $cov(n,\varepsilon,\alpha,K)$ and so $\mbox{sep}(n, \varepsilon, \alpha,K )\leq\mbox{cov}(n, \varepsilon, \alpha,K )$.

Finally, for the second inequality, let $B\subseteq K$ be an $(n, \varepsilon)$-separated set of maximal cardinality (which exists because $\mbox{sep}(n, \varepsilon, \alpha,K )\leq\mbox{cov}(n, \varepsilon, \alpha,K )$ and $\mbox{cov}(n, \varepsilon, \alpha,K )$ is finite). 
Notice that, for every $y\in K\backslash B$, there exists $x\in B$ such that $I_n(x)=I_n(y)$, since otherwise, if there exists $y\in K\backslash B$ such that $I_n(x)\neq I_n(y)$ for all $x\in B$, then $B\cup \{y\}$ is an $(n,\varepsilon)$-separated set, what contradicts the maximality of $B$. Furthermore, for all $y\in K\backslash B$, there exists $x\in B$ such that $d_n(y,x)<\varepsilon$, since otherwise $B\cup \{y\}$ is again an $(n,\varepsilon)$-separated set, contradicting the maximality of $B$. We conclude that $B$ is an
 $(n, \varepsilon)$-spanning set, and hence $\mbox{span}(n, \varepsilon, \alpha,K )\leq \mbox{sep}(n, \varepsilon, \alpha,K )$.

\end{proof}

\begin{remark} It follows from the third inequality of the previous lemma that each $(n,\varepsilon)$-separated set of $K$ is finite.
\end{remark}

\begin{definition}\label{defentropy}Let $\al=(\{X_n\}_{n\in \Z},\{\al_n\}_{n\in \Z})$ be a partial action on a metric space $X$, where each $X_n$ is clopen.
\begin{enumerate}
    
\item For each $K\subseteq X$ compact define
$$
h_{\varepsilon}^+(\alpha,K,d)=\limsup_{n\to\infty}\frac{1}{n}\log \mbox{sep}(n,\varepsilon,\alpha,K), \text{ and}$$
$$
h_{\varepsilon}^-(\alpha,K,d)=\limsup_{n\to-\infty}\frac{1}{-n}\log \mbox{sep}(n,\varepsilon,\alpha,K).$$
\item For each $K\subseteq X$ compact define
$$
h_{d}^+(\alpha,K)=\lim_{\varepsilon\to0}h^+_{\varepsilon}(\alpha,K,d) \text{\  and \ } h_{d}^-(\alpha,K)=\lim_{\varepsilon\to0}h^-_{\varepsilon}(\alpha,K,d).
$$
\item 
Define the partial topological entropy of the partial action $\alpha$ as
\begin{equation}
\hbar_{d}(\alpha)=\max\left\lbrace \sup_{K\subseteq X}h^+_{d}(\alpha,K), \sup_{K\subseteq X}h^-_{d}(\alpha,K)\right\rbrace.
\end{equation}
\end{enumerate}
\end{definition}

\begin{remark}\label{intervice} 
Notice that the quantity $\mbox{sep}(n,\varepsilon,\alpha,K)$ increases monotonically as $\varepsilon$ decreases, and so $h_{\varepsilon}^+(\alpha,K,d)$ and $h_{\varepsilon}^-(\alpha,K,d)$ do as well. Therefore the limits in the second item of the previous definition do exist. Furthermore, the inequalities in Lemma \ref{lemma:key-lemma}
imply that equivalent definitions of $\hbar_{d}(\alpha)$ can be obtained
 if we replace $\mbox{sep}(n, \varepsilon, \alpha,K)$ by
 $\mbox{span}(n,\varepsilon,\alpha,K)$ or
$\mbox{cov}(n,\varepsilon,\alpha,K)$ in the expressions above.
\end{remark}

If $\al=(\{X_n\}_{n\in \ZZ},\{\al_{n}\}_{n\in \ZZ})$ is a partial action, it follows directly from the definition of entropy that $\hbar_{d}(\al) = \hbar_{d}(\al^{-1})$, where $\al^{-1}$ is the inverse partial action, that is, $\al^{-1}=(\{X_{-n}\}_{n\in \ZZ},\{\al^{-1}_{n}\}_{n\in \ZZ})$. Furthermore, if the partial action $\al$ is a full action of $\Z$ then it has a generator $\al_1 = T$, which is a 
 continuous homeomorphism of $X$, and the entropy above defined coincides with the the usual topological entropy of the map $T$ (for more details see \cite{RB} and references therein).

Notice that it is important to consider both the forward and backward orbits of partial actions. This becomes clear in the following example.

\begin{example}
Let $X=\R$ and consider the action given by the homeomorphism $\alpha(x) = \frac{1}{2} x$. Then  $\sup_{K\subseteq X}\hbar^+_{d} =0$ and $ \sup_{K\subseteq X}\hbar^-_{d} = log 2$. Hence $ \hbar_{d}(\alpha)= log 2$. But, if the entropy was defined just in terms of forward orbits, the entropy would be zero.
\end{example}

Since we are defining a topological entropy, we would like our definition to be independent of the metric generating the topology. To develop on this topic, we need the following definition. 

\begin{definition}
We say that metrics $d_1$ and $d_2$ on $X$ are continuously equivalent if $Id_X:(X,d_1)\to (X,d_2)$ and $Id_X:(X,d_2)\to (X,d_1)$ are continuous. If both maps are uniformly continuous then we say that the metrics are uniformly equivalent.
\end{definition}

Notice that if $d_1$ and $d_2$ are 
continuously equivalent then a map $f:(X,d_1)\to (X,d_1)$ is continuous if and only $f:(X,d_2)\to (X,d_2)$ is continuous. Therefore, if $\alpha$ is a partial action on $(X,d_1)$, so each $\alpha_n$ is a homeomorphism, 
then $\alpha$ is partial action of $(X,d_2)$ and vice-versa.

 \begin{lemma}\label{invariance by the metric} Let $(\{\alpha_n\}_{n\in \ZZ},\{X_n\}_{n\in \ZZ})$ be a partial action on $(X,d)$ 
 and let $d'$ be a metric on $X$ such that $d$ and $d'$ are 
 uniformly equivalent metrics on 
 $X$. Then $\hbar_{d}(\al)=\hbar_{d'}(\al)$.
\end{lemma}
\begin{proof}

Let $K\subseteq X$
compact. 
Given $\varepsilon_1>0$, choose $\varepsilon_2\geq\varepsilon_3>0$, satisfying $\varepsilon_2<\varepsilon_1$ and such that
$d(x,y)\leq\varepsilon_1$ whenever $d'(x,y)\leq\varepsilon_2$, and
$d'(x,y)\leq\varepsilon_2$ whenever $d(x,y)\leq\varepsilon_3$, for every $x, y \in K$.  Then an $(n,\varepsilon_2)$-spanning set for $K$ with respect to $d'$ is
$(n,\varepsilon_1)$-spanning for $K$ with respect to $d$. Hence
$$
\mbox{span}(n,\varepsilon_1,\al,K)_d\leq\mbox{span}(n,\varepsilon_2,\al,K)_{d'},$$ where $\mbox{span}(n,\varepsilon_1,\al,K)_d$ is the minimum cardinality of the $(n,\varepsilon_1)$-spanning sets of $K$ with respect to the metric $d$.

Similarly
$$
\mbox{span}(n,\varepsilon_2,\al,K)_{d'}\leq\mbox{span}(n,\varepsilon_3,\al,K)_d.
$$
Applying Definition \ref{defentropy} and Remark \ref{intervice} and letting $n\to\infty$ we get
$$
h_{\varepsilon_1}^+(\alpha,K,d)\leq h_{\varepsilon_2}^+(\alpha,K,d')\leq h_{\varepsilon_3}^+(\alpha,K,d).
$$ 
Letting $\varepsilon_1\to 0$, 
$$
h_{d}^+(\al,K)=h_{d'}^+(\al,K),
$$
and a similar argument shows that  $h_{d}^-(\al,K)=h_{d'}^-(\al,K)$. Hence $\hbar_d(\alpha)=\hbar_{d'}(\alpha)$.
\end{proof}


\begin{remark}\label{vamosleao} 
Notice that uniform equivalence of the metrics is crucial in the lemma above. For an example in which the entropy is not preserved by continuously equivalent metrics, see \cite[Remark~15, pag. 171]{Walters}.
\end{remark}

We prove below that entropy is an invariant for uniform equivalence of partial actions (by uniform equivalence we mean that the map implementing the equivalence, and its inverse, are uniformly continuous).



\begin{corollary}
Let $\al$ and  $\beta$ be partial actions on the
metric spaces $(X,d)$ and $(Y,\ro)$, respectively. Let $\phi:X\rightarrow Y$ be an equivalence (as in Definition~\ref{pqfotocoordenadores}) between $\al$ and $\beta$ such that both $\phi$ and $\phi^{-1}$ are uniformly continuous.
Then $\hbar(\alpha)=\hbar(\beta)$.
\end{corollary}
\begin{proof}
 Define a metric $d'$ on $Y$ by setting $d'(y_1,y_2)=d(x_1,x_2)$, where $\phi(x_i)=y_i$ for $i=1,2$. It is clear that $d'$ is
a metric on $Y$ that is uniformly equivalent to $\ro$. Moreover, $\phi$ is an isometry from $(X,d)$ to $(Y,d')$, $I_n(x)=I_n(\phi(x))$ for all $x$ and all $n$, and $d_n(x,y)=d'_n(\phi(x),\phi(y))$ for all $x,y \in X$. Since the entropy is independent of uniformly equivalent metrics, it follows that $\hbar_{d}(\al)=\hbar_{d}(\beta)$.
\end{proof}

\subsection{The entropy of a homeomorphism}
 We are now in a position to answer a question raised in the introduction:
how to define a notion of topological entropy
for a homeomorphism $f \colon A \to f(A)\subset X$.  The main obstacle is the fact that, in general,
it is not possible to consider full orbits for points in $A$;
on the other hand, a partial action can be defined and its
corresponding partial topological entropy makes sense.

  Take $A$ a clopen subset of $X$
 and $f \colon A \to f(A)\subset X$ a homeomorphism such that $f(A)$ is also clopen in $X$. This induces a partial action $\al=(\{X_{-n}\},\{f^n\})_{n\in\ZZ}$ of $\ZZ$ by defining $X_{-n} =\text{dom}(f^n)$, which is clopen, for all $n\in\mathbb Z$. Then we define the entropy of the homeomorphism $f$ as the partial entropy of this particular partial action (see \cite{Ruy1,GGS} for the proof that this is a partial action).

\begin{remark} The reader can verify that if there exists $N\in\mathbb N$ such that $X_n=\emptyset$, for all $n\in\mathbb Z$  with $|n|\geq N$, then the entropy
 of the homeomorphism is zero. In particular, this is the case when $A \cap f(A) = \emptyset$,
 a situation where there is no dynamics in the usual sense.
\end{remark}

An interesting example of the above definition is the case of minimal homeomorphisms of the Cantor set. In this context, using Bratteli-Vershik models, Boyle and Handelman answered (in the negative) the question of whether entropy is an invariant for strong orbit equivalence  (see \cite{BH}) and computed entropy for examples, including the dyadic adding machine. The restrictions of minimal homeomorphisms of the Cantor set to clopen sets forms, therefore, a class of examples where partial topological entropy can be applied. Due to the notation required and in light of the space required, we leave it to the interested reader to further develop the topic.


\section{Properties}

In this section, we establish some properties of the partial entropy. 
We suppose that $X$ is a metric space with a given metric $d$, $\al=(\{X_n\}_{n\in\ZZ},\{\al_n\}_{n\in\ZZ})$ is a partial action with each $X_n$ clopen, and $Y \subseteq X$ is a $\alpha$-invariant subset of $X$, that is, $$\al_n(Y\cap X_{-n}) \subseteq Y\cap X_n$$ for all $n \in \ZZ$. Let $\varphi$ be the restriction of $\alpha$ to $Y$, that is, \[\varphi=\alpha|_{Y}:=(\{X_n\cap Y\}_{n\in\ZZ},\{\al_n\}_{n\in\ZZ}),\] which is a $\ZZ$ partial action in $Y$.

\begin{proposition}\label{subaction}
Let $X,Y$, $\alpha$ and $\varphi$ be as above. Then $\hbar_{d}(\varphi)\leq \hbar_{d}(\alpha)$.
\end{proposition}
\begin{proof}
 First note that each $K\subseteq Y$ compact in $Y$ is compact in $X$. Moreover, for $K\subseteq Y$ compact and  $n\in \Z$, $sep(n,\varepsilon, \alpha, K)=sep(n,\varepsilon, \varphi, K)$. Therefore $h_\varepsilon^+(\alpha, K, d)=h_\varepsilon^+(\varphi, K, d)$, and $h_\varepsilon^-(\alpha, K, d)=h_\varepsilon^-(\varphi, K, d)$. Since there are more compact sets in $X$ than in $Y$ then $\hbar_d(\alpha)\geq \hbar_d(\varphi)$.
 
\end{proof}

 \begin{proposition}\label{prop:lower-bound-entropy}
Let $\alpha=(\{\alpha_n\}_{n\in \Z},\{X_n\}_{n\in \Z})$ be a $\Z$ partial action on a metric space $X$, with each $X_n$ clopen, and let $A_i$, for $i = 1, \ldots, k$, be $\alpha$-invariant closed subsets of $X$, whose union is $X$. Then 
$$
\hbar_{d}( \alpha ) =\max\{\hbar_{d}( \alpha|_{A_i}):1\leq i\leq k\}.
$$
\end{proposition}
\begin{proof} From Proposition~\ref{subaction} we get that $h_{d}( \alpha ) \geq\max h_{d}( \alpha|_{A_i}).$ We show the converse inequality.
Let $K\subset X$ be a compact set and  $Y\subset K$ be a $(n,\varepsilon)$-separated
set in $K$. Let $K_i=K\cap A_i$, which is compact, since $A_i$ is closed, and let $Y_i=Y\cap A_i$. Notice that $Y_i$ is $(n,\varepsilon)$-separated in $K_i$.
Then 
$$
\mbox{sep}(n,\varepsilon,\alpha,K)\leq \sum\limits_{i=1}^k\mbox{sep}(n,\varepsilon,\alpha|_{A_i},K_i)
\leq k\max_{1\leq i\leq k}\mbox{sep}(n,\varepsilon,\alpha|_{A_i},K_i).
$$ 
Hence we have that
\begin{align*}
h_\varepsilon^+(\alpha,K,d)&=\limsup_{n\to\infty}\frac{1}{n} \log \mbox{sep}(n,\varepsilon,\alpha,K)\\
                           &\leq\limsup_{n\to\infty}\frac{1}{n} \log(k\max_{1\leq i\leq k}\mbox{sep}(n,\varepsilon,\alpha|_{A_{i}},K_{i}))\\
                           &=\max_{1\leq i\leq k}\limsup_{n\to\infty}\frac{1}{n} \log\mbox{sep}(n,\varepsilon,\alpha|_{A_{i}},K_{i})\\
                           &=\max_{1\leq i\leq k}h_\varepsilon^+(\alpha|_{A_i}, K_i,d)\\
                           &\leq \max_{1\leq i\leq k}h_d^+(\alpha|_{A_i}, K_i)\\
                           &\leq \max_{1\leq i\leq k}\hbar_d(\alpha|_{A_i})
\end{align*}
Letting $\varepsilon\to0$ we obtain $h_d^+(\alpha, K)\leq \max\limits_{1\leq i\leq k}\hbar_d(\alpha|_{A_i})$. Similarly one shows that $h_d^-(\alpha, K)\leq \max\limits_{1\leq i\leq k}\hbar_d(\alpha|_{A_i})$. Therefore
$$\hbar_d(\alpha)\leq \max\limits_{1\leq i\leq k}\hbar_d(\alpha|_{A_i}).$$
\end{proof}

\subsection{An upper bound for the partial entropy}
In order to compare the partial topological entropy of a partial action
and the topological entropy of its globalization, we need to have some
information about the topology of the quotient space. More precisely,
there is no guarantee that the quotient space obtained in the  globalization
of a partial action is a metric space. Actually, \cite[Proposition~2.10]{MR1957674} shows that $X^e$ is a Hausdorff space if, and only if,
the graph of $\alpha $ is a closed subset of $G\times X\times X$, and in \cite[Proposition~2.1]{EGG} a characterization is given for partial actions on second countable, Hausdorff spaces.

In general the partial topological entropy
is smaller than the topological entropy of the globalization, as illustrated in the example below.

\begin{example}
Let $X=\{0,1\}^{\mathbb Z}$ be the Bernoulli space and consider on $X$ the shift map
$\sigma(x)_i=x_{i+1}$. We can think of $\sigma:X\to X$ as a $\mathbb Z$-action. It is well
known that the topological entropy of $\sigma$ is $\log2$. Let $Y$ be the cylinder
$[01]$ and define a partial action $(\alpha,\mathbb{Z}, X)$ given by the restriction of $\sigma$
to $Y$. As $\sigma ([01])=[1]$, we have that $\alpha_{n}(Y)\cap Y=\emptyset$ for all $n\in\mathbb N$,
 and it implies that $h_{d}(\alpha)=0$. 
\end{example}

\begin{theorem}
Let $\al=(\{X_n\}_{n\in\ZZ},\{\al_n\}_{n\in\ZZ})$ be a partial action, with each $X_n$ clopen, and let $\al^e:\ZZ\times X^e\to X^e$ be
its globalization. Suppose that $X^e$ is metrizable by a metric $d^e$, which restricted to $X$
is continuously equivalent to the metric $d$ of $X$. If
$h_{d^e}(\al^e)$ is the
topological entropy of $\al^e$, then
$$
h_{d}(\al)\leq h_{d^e}(\al^e).
$$
\end{theorem}

\begin{proof}
  The proof follows from Proposition \ref{prop:lower-bound-entropy}, since $X\subset X^e $,
$\al^e(X)=X$, and $\alpha$ can be seen as the restriction of the partial action $\al^e$.
\end{proof}

\section{Entropy and nonwandering sets}

In \cite{BowenWandering} Bowen 
showed that for a global $\mathbb Z$-action on a compact metric space the topological entropy is concentrated
on the nonwandering set. This result, in a certain sense, shows that if one intends to study 
how chaotic is the dynamics of a system, the nonwandering set is the right place to do it. In what follows we will prove an analogous result for semi-saturated partial actions. We recall the definition of semi-saturated partial actions of $\Z$ below.

\begin{definition} \cite[Definition~4.9]{Ruy}
A partial action $\al=(\{X_n\}_{n\in\ZZ},\{\al_n\}_{n\in\ZZ})$ is called semi-saturated if $\alpha_{m+n}=\alpha_m\alpha_n$ for all $m,n\geq 0$ and $m,n\leq 0$. 
\end{definition}

\begin{remark}
 Semi-saturated partial actions are related to Quigg and Raeburn's notion of multiplicativity, see \cite{QR}, and appear often in the study of partial actions, see \cite{Toke, CW, Ruy, GR, GR1} for example.
\end{remark}

We define wandering and nonwandering points as follows.

\begin{definition} Let $(\{X_n\}_{n\in \Z}, \{\alpha_n\}_{n\in \Z})$ be a partial action. We say that a point $x\in X$ is wandering if there exists an open set $U\subseteq X$, with $x\in U$, such that $\alpha_n(U\cap X_{-n})\cap U=\emptyset$ for all $n\in \Z^*$.

A point $x\in X$ is nonwandering if it is not  wandering, that is, if for each open set $U\subseteq X$, with $x\in U$, there exists some $n\in \Z^*$ such that $\alpha_n(U\cap X_{-n})\cap U\neq \emptyset$. We denote by $\Omega(\alpha)$ the set of all nonwandering points.
\end{definition}

We want to restrict a partial action to the set of the nonwandering points. For this, we need the following result.

\begin{proposition}\label{prop:omega-invar}
The set $\Omega(\al)$ is $\alpha$-invariant and closed.
\end{proposition}

\begin{proof} Let $x\in X\setminus \Omega(\alpha)$, and let $U\subseteq X$ be an open set with $x\in U$ and $\alpha_n(U\cap X_{-n})\cap U=\emptyset$. Then each element $y\in U$ is also wandering. Therefore $X\setminus \Omega(\alpha)$ is open, and so $\Omega(\alpha)$ is closed.

Now we show that $\Omega(\alpha)$ is $\alpha-$ invariant. Let $x\in \Omega(\alpha)\cap X_{-n}$ with $n\in \Z$, and let $V\subseteq X$ be an open set containing $\alpha_n(x)$. Let $V'=V\cap X_n$ and define $U=\alpha_{-n}(V')$, which is an open set containing $x$. Then there exists $m\in \Z^*$ such that $\alpha_m(U\cap X_{-m})\cap U\neq \emptyset$, that is, $\alpha_m(\alpha_{-n}(V')\cap X_m)\cap \alpha_{-n}(V')\neq \emptyset$ and so $\alpha_n[\alpha_m(\alpha_{-n}(V')\cap X_{-m})]\cap V'\neq \emptyset$. Notice that
\begin{align*}
    \alpha_n[\alpha_m(\alpha_{-n}(V')\cap X_{-m})]              &=\alpha_n[\alpha_m(\alpha_{-n}(\alpha_n(\alpha_{-n}(V')\cap X_{-m})))]\\
    &\subseteq\alpha_n[\alpha_m(\alpha_{-n}(V'\cap X_{n-m}))]\\
    &=\alpha_{m}(V'\cap X_{-m-n})\\ &\subseteq\alpha_{m}(V'\cap X_{-m}).
\end{align*}

As $ \alpha_n[\alpha_m(\alpha_{-n}(V')\cap X_{-m})]\cap V'\neq \emptyset$ and $ \alpha_n[\alpha_m(\alpha_{-n}(V')\cap X_{-m})]\subseteq \alpha_m(V'\cap X_{-m})
$, 
we obtain that
$$\alpha_m(V'\cap X_{-m})\cap V'\neq \emptyset.$$ 
Since $V'\subseteq V$, it follows that

$$\alpha_m(V\cap X_{-m})\cap V\neq \emptyset.$$ 
Therefore $\alpha_n(x)\in \Omega(\alpha)$.
\end{proof}

\begin{remark}
In the definition of the set of nonwandering points of a partial action, and in the proposition above, it is not necessary to assume that the domains of the partial action are clopen.
\end{remark}

Since $\Omega(\alpha)$ is $\alpha$-invariant, it makes sense to consider the restriction $\alpha_{|_{\Omega(\alpha)}}$ (see the paragraph above Proposition~\ref{subaction}). The following theorem is the main result of this section. 

\begin{theorem}\label{nonwandering}
	Let $\al=(\{X_n\}_{n\in\ZZ},\{\al_n\}_{n\in\ZZ})$ be a semi-saturated partial action, with each $X_n$ clopen, on a  compact	metric space $(X,d)$. Then $h_{d}(\al)=h_{d}(\al|_{\Omega(\al)}).$
\end{theorem}
\begin{proof} 
It follows from Proposition \ref{subaction} that $h_d(\alpha)\geq h_d(\alpha_{|_{\Omega(\alpha)}})$.
We prove the reverse inequality. Let $0\neq n\in \Z$ and $\varepsilon>0$. Let $A$ be an $(n,\varepsilon)$-spanning set of minimum cardinality contained in $\Omega(\alpha)$. For each $y\in A$ let $$U(y,n,\varepsilon)=\{x\in X:I_n(x)=I_n(y) \text{ and } d_n(x,y)<\varepsilon\},$$ which is open, and notice that $(U(y,n,\varepsilon))_{y\in A}$ is an open cover of $\Omega(\alpha)$, since $A$ is an $(n,\varepsilon)$-spanning set of $\Omega(\alpha)$. Let $U=\bigcup\limits_{y\in A}U(y,n,\varepsilon)$. Notice that $U^c=X\setminus U$ is compact, and $A$ is an $(n,\varepsilon)$-spanning set of $U$.

{\it Claim 1: For each $N\in \Z$ there exists $0<\beta\leq \varepsilon$ such that 
$$
\alpha_i(U(x,N,\beta)\cap X_{-i})\cap U(x,N,\beta)=\emptyset
$$ for each $i\neq 0$ and $x\in U^c$.} 

To prove this claim, suppose that for each $\beta>0$ there exists $x\in U^c$ and $i\in \Z^*$ such that $\alpha_i(U(x,N,\beta)\cap X_{-i})\cap U(x,N,\beta)\neq \emptyset$.
Then, for each $m\in \N$, let $x_m\in U^c$ and $i_m\in \Z^*$ be such that  
$$
\alpha_{i_m}(U(x_m,N,\frac{1}{m})\cap X_{-i_m})\cap U(x_m,N,\frac{1}{m})\neq \emptyset.
$$ 
Since $U^c$ is compact, there exists a subsequence $(x_{m_k})_{k\in \N}$ of $(x_m)_{m\in \N}$ such that $x_{m_k}\rightarrow z\in U^c$. Let $V\subseteq X$ be an open set with $z\in V$. So, there exists $k_0$ such that $B(x_{m_{k_0}},\frac{1}{m_{k_0}})\subseteq V$ and, as $U(x_{m_{k_0}},N,\frac{1}{m_{k_0}})\subseteq B(x_{m_{k_0}},\frac{1}{m_{k_0}})$, we have that $U(x_{m_{k_0}},N,\frac{1}{m_{k_0}})\subseteq V$. Since $$
\alpha_{i_{m_{k_0}}}(U(x_{m_{k_0}},N,\frac{1}{m_{k_0}})\cap X_{-i_m})\cap U(x_{m_{k_0}},N,\frac{1}{m_{k_0}})\neq \emptyset
$$ and  $U(x_{m_{k_0}},\frac{1}{m_{k_0}})\subseteq V$, we obtain
$$
\alpha_{i_{m_{k_0}}}(V\cap X_{-i_{m_{k_0}}})\cap V\neq \emptyset.
$$ 
Hence $z$ is nonwandering, which is impossible since the elements of $U^c$ are all wandering points. This proves Claim 1.

Let $B$ be an $(n,\beta)$-spanning set of minimum cardinality of $U^c$, and define $C=A\cup B$, which is an $(n,\varepsilon)$-spanning set of $X$. 

Let $l\in \N$ and define  $\varphi_l:X\rightarrow (C\cup \{\star\})^l$, where $\star$ is just a symbol, by $\varphi_l(x)=(y_0,..,y_{l-1})$ where $y_i$ is chosen as follows. 

Case $n>0$: for $0\leq i\leq l-1$, if $x\in X_{-in}$ and $\alpha_{in}(x)\in U$, then there exists some $u_i\in A$ such that $d_n(\alpha_{in}(x), u_i)<\varepsilon$ and $I_n(\alpha_{in}(x))=I_n(u_i)$. Let $y_i:=u_i$ in this case; if $x\in X_{-in}$ and $\alpha_{in}(x)\in U^c$, then there exists some $v_i\in B$ such that $d_n(\alpha_i(x), v_i)<\varepsilon$ and $I_n(\alpha_{in}(x))=I_n(v_i)$. Let $y_i:=v_i$ in this case; if $x\notin X_{-in}$ define $y_i:=\star$. 

Case $n<0$: for $0\leq i\leq l-1$, if $x\in X_{in}$ and $\alpha_{-in}(x)\in U$, then there exists some $u_i\in A$ such that $d_n(\alpha_{-in}(x), u_i)<\varepsilon$ and $I_n(\alpha_{-in}(x))=I_n(u_i)$. Let $y_i:=u_i$ in this case; if $x\in X_i$ and $\alpha_{-in}(x)\in U^c$, then there exists some $v_i\in B$ such that $d_n(\alpha_{-in}(x), v_i)<\varepsilon$  and $I_n(\alpha_{-in}(x))=I_n(v_i)$. Let $y_i:=v_i$ in this case; if $x\notin X_{-in}$ define $y_i:=\star$.

Before we continue with the proof, we notice that the definition of the function $\varphi_l$ depends on the choice of the $u_i$ and the $v_i$ (which might not be unique). Nevertheless, independently of the choice of the $u_i$ and $v_i$ (and hence of $\varphi_l)$, the key point is that any such function is injective on certain separated sets (as we prove below). We, therefore, fix a $\phi_l$ from now on.

{\it Claim 2: For each $x\in X$, each $y\in B$ appears at most one time in $\varphi_l(x)$.}

We prove the case $n>0$. The case $n<0$ is analogous. Suppose that there exists some $y\in B$ appearing two times in $\varphi_l(x)$. Then $y_i=y=y_j$ for some $i,j\in \{0,...,l-1\}$. As $\alpha_{in}(x)\in U(y,n,\beta)$ and $\alpha_{jn-in}(\alpha_{in}(x))=\alpha_{jn}(x)\in U(y,n,\beta)$, it follows that $\alpha_{jn-in}(U(y,n,\beta)\cap X_{in-jn})\cap U(y,n,\beta)\neq\emptyset$, which is impossible by Claim 1.

{\it Claim 3: For $n>0$ and $0<m\leq nl$, if $E$ is an $(m,2\varepsilon)$-separated set then the map $\varphi_l$ is injective on $E$.}

Let $m\leq ln$ and $x,x'\in X$ be $(m,2\varepsilon)$-separated points such that $\varphi_l(x)=\varphi_l(x')$. In particular, the first coordinate of  $\varphi_l(x)$ equals to the first coordinate of $\varphi_l(x)$. This implies that there exists some $y\in A\cup B$ such that $I_n(x)=I_n(y)=I_n(x')$, $d_n(x,y)\leq\varepsilon$, and $d_n(x',y)\leq \varepsilon$. For $i\in \{0,...,l-1\}$, we have that $x\in X_{-in}$ if and only if $x'\in X_{-in}$. Since $\varphi_l(x)=\varphi_l(x')$, for $i\in \{0,...,l-1\}$ with $x,x'\in X_{-in}$, it follows that $d_n(\alpha_{in}(x),y_i)< \varepsilon$ and $d_n(\alpha_{in}(x'),y_i)<\varepsilon$, from where $d(\alpha_{in}(x),y_i)< \varepsilon$ and $d(\alpha_{in}(x'),y_i)<\varepsilon$. Therefore $d(\alpha_{in}(x),\alpha_{in}(x'))<2\varepsilon$. Moreover, from the definition of $\varphi_l$,  $I_n(\alpha_{in}(x))=I_n(\alpha_{in}(x'))$. Now, take  $0<t<n$  and suppose that $x,x'\in X_{-in-t}$, for some $0\leq i < l$. As $\alpha$ is semi-saturated,  $\alpha_{t+in}(x)=\alpha_t(\alpha_{in}(x))$  and so $t\in I_n(\alpha_{in}(x))$. Using the fact that $I_n(\alpha_{in}(x))=I_n(\alpha_{in}(x'))=I_n(y_i)$ we obtain  
\begin{align*}
d(\alpha_{in+t}(x),\alpha_{in+t}(x'))&\leq d(\alpha_{in+t}(x),\alpha_t(y_i))+d(\alpha_t(y_i), \alpha_{in+t}(x'))\\
                                     & =d(\alpha_t(\alpha_{in}(x)),\alpha_t(y_i))+d(\alpha_t(y_i), \alpha_t(\alpha_{in}(x'))\\
                                     &\leq d_n(\alpha_{in}(x),y_i)+d_n(y_i, \alpha_{in}(x')\\
                                     &<2\varepsilon.
\end{align*}
From the last inequality, and from the fact that $I_n(x')=I_n(x)$, we conclude that $d_m(x,x')<2\varepsilon$, which is a contradiction, since $x,x'$ are $(m,2\varepsilon)$-separated.

{\it Claim 4: For $n<0$ and $ln\leq m<0$, the map $\varphi_l$ is injective on the set of all $(m,2\varepsilon)$-separated points.}

The proof of this claim is analogous to the previous one.

{\it Claim 5: Let $n\in \Z$, $l\in \N$  and let $m\in \Z$ such that $0<m\leq nl$ if $n>0$, or, $ln\leq m<0$ if $n<0$. Let $q=span(n,\beta, \alpha, U^c)$ and $p=span(n,\varepsilon,\alpha_{|_{\Omega(\alpha)}}, \Omega(\alpha))$, and let $E$ be an $(m,2\varepsilon)$-separated set  of maximum cardinality in $X$. Then it holds that $\#(\varphi_l(E))\leq (q+1)!l^qp^l$. }

For each $0\leq j$ let $I_j$ be the subset of the elements $\varphi_l(x)=(y_0,...,y_{l-1})$ of $\varphi_l(E)$ with exactly $j$ entries $y_i$ in $B$. Notice that for $j>q$ it follows from Claim~2 that $I_j=\emptyset.$ Notice that there are $\frac{q!}{j!(q-j)!}$ ways to chose $j$ elements in $B$ and there are $\frac{l!}{(l-j)!}$ ways of arranging each choice. Moreover, there are at most $p^{l-j}\leq p^l$ ways to choose the remaining elements. Therefore $$\#I_j\leq \frac{q!}{j!(q-j)!}\frac{l!}{(l-j)!}p^l.$$  
 Since $\frac{q!}{j!(q-j)!}\leq q!$ and $\frac{l!}{(l-j)!}\leq l^q$
 $$\#I_j\leq q!l^q p^l,$$ and so 
 
$$\#\varphi_l(E)=\sum\limits_{j=0}^p \#I_j\leq \sum\limits_{j=0}^q q!l^q p^l\leq (q+1)!l^q p^l.$$

By Claim 3 and Claim 4, $\varphi_l$ is injective on $E$, and so $\# E=\#\varphi_l(E)$, and therefore $\#E\leq (q+1)!l^q p^l$.

Now fix $n>0$, let $p$ and $q$ as in Claim 5, let $m>0$ and chose $l_m\in \N$ such that $n(l_m-1)<m\leq nl_m$ and let $E$ as in Claim 5.

Then 

\begin{align*}
\frac{\log(sep(m,2\varepsilon, \alpha, X))}{m}
            &=\frac{\log(\# E)}{m}\leq \frac{\log((q+1)!l_m^qp^{l_m})}{m}\\
            &\leq \frac{\log((q+1)!l_m^qp^{l_m})}{n(l_m-1)}\\ 
            &\leq\frac{1}{n(l_m-1)}(\log(q+1)!+q\log(l_m)+l_m\log(p)).
\end{align*}
 Notice that if $m\rightarrow \infty$, $l_m\rightarrow \infty$ and then 
 \begin{align*}
    h_{2\varepsilon}^+(\alpha,X,d)
            &=\limsup\limits_{m\rightarrow \infty}\frac{\log(sep(m,2\varepsilon, \alpha, X))}{m}\\
            &\leq \limsup\limits_{m\rightarrow \infty}\frac{1}{n(l_m-1)}(\log(q+1)!+q\log(l_m)+l_m\log(p))\\
            &=\frac{log(p)}{n}\\
            &=\frac{\log(span(n,\varepsilon, \alpha_{|_{\Omega(\alpha)}}, \Omega(\alpha)))}{n}\leq \frac{\log(sep(n,\varepsilon, \alpha_{|_{\Omega(\alpha)}}, \Omega(\alpha)))}{n}. 
 \end{align*}
   So, it holds that $$h_{2\varepsilon}^+(\alpha,X,d)\leq\frac{\log(sep(n,\varepsilon, \alpha_{|_{\Omega(\alpha)}}, \Omega(\alpha)))}{n}$$
   for each $n\in \N$, and then 
 \begin{align*}
     h_{2\varepsilon}^+(\alpha,X,d)&\leq\limsup\limits_{n\in \infty}\frac{\log(sep(n,\varepsilon, \alpha_{|_{\Omega(\alpha)}}, \Omega(\alpha)))}{n}\\
                                   &=h_\varepsilon^+(\alpha_{|_{\Omega(\alpha)}},\Omega(\alpha), d),
 \end{align*}
 for each $\varepsilon >0$. Therefore 
 $$
 h_d^+(\alpha, X)\leq h_d^+(\alpha_{|_{\Omega(\alpha)}},\Omega(\alpha)).
 $$
 
Analogously one shows that $$h_d^-(\alpha, X)\leq h_d^-(\alpha_{|_{\Omega(\alpha)}},\Omega(\alpha)),$$

 which guarantees that $h_d(\alpha)\leq h_d(\alpha_{|_{\Omega(\alpha)}})$ and ends the proof of the theorem.
\end{proof}

We finish the paper applying the above theorem to compute the partial entropy of the action defined in Example~\ref{siri}. 

\begin{example} 
Let $\alpha$ be the partial action of Example~\ref{siri}. It is straightforward to check that $\alpha$ is semi-saturated. Notice that the only nonwandering point in $X$ is $0^\infty:= 0 0 0 0 \ldots$, so that  $\Omega(\alpha)=\{0^\infty\}$. Hence, by Theorem~\ref{nonwandering}, we have that $h_d(\alpha)=0$.
\end{example}

 Alexandre Baraviera, Departamento de Matem\'atica Pura e Aplicada - IME,
 Universidade Federal do Rio Grande do Sul, Porto Alegre, 91509-900, Brazil
 
 Email: atbaraviera@gmail.com

\vspace{1pc}

Ruy Exel, Departamento de Matem\'atica, Universidade Federal de Santa Catarina, Florian\'opolis, 88040-900, Brazil.

Email: exel@mtm.ufsc.br

\vspace{1pc}

Daniel Gon\c{c}alves, Departamento de Matem\'atica, Universidade Federal de Santa Catarina, Florian\'opolis, 88040-900, Brazil.

Email: daemig@gmail.com

\vspace{1pc}

 Fagner B. Rodrigues, Departamento de Matem\'atica Pura e Aplicada - IME,
 Universidade Federal do Rio Grande do Sul, Porto Alegre, 91509-900, Brazil
 
 Email: fagnerbernardini@gmail.com

\vspace{1pc}
Danilo Royer, Departamento de Matem\'atica, Universidade Federal de Santa Catarina, Florian\'opolis, 88040-900, Brazil.

Email: danilo.royer@ufsc.br

\end{document}